\documentclass{amsart}
\pdfoutput=1

\usepackage{amssymb}
\usepackage{microtype}
\usepackage{tikz}
\usepackage{booktabs}
\usepackage[pdftex,colorlinks,citecolor=black,linkcolor=black,urlcolor=black,bookmarks=false]{hyperref}
\def\MR#1{\href{http://www.ams.org/mathscinet-getitem?mr=#1}{MR#1}}
\def\arXiv#1{arXiv:\href{http://arXiv.org/abs/#1}{#1}}
\usepackage{doi}

\usepackage{dcolumn}
\usepackage{color, colortbl}
\definecolor{Gray}{gray}{0.95}
\newcolumntype{d}[1]{D{.}{.}{#1}} 

\usepackage{pgfplots}
\pgfplotsset{compat=1.14}
\pgfmathdeclarefunction{lp_f}{1}{\pgfmathparse{(1-#1^2)*exp(-#1^2)}}
\pgfmathdeclarefunction{lp_fhat}{1}{\pgfmathparse{(17.49341833*#1^2+0.8862269255)*exp(-9.869604404*#1^2)}}

\newtheorem{theorem}{Theorem}[section]
\newtheorem{proposition}[theorem]{Proposition}

\numberwithin{equation}{section}
\numberwithin{figure}{section}
\numberwithin{table}{section}

\newcommand{\R}{\mathbb{R}}
\newcommand{\C}{\mathbb{C}}
\newcommand{\Q}{\mathbb{Q}}
\newcommand{\Z}{\mathbb{Z}}
\newcommand{\N}{\mathbb{N}}
\newcommand{\M}{\mathcal{M}}
\newcommand{\SL}{\textup{SL}}
\newcommand{\GL}{\textup{GL}}
\newcommand{\vol}{\mathop\textup{vol}}
\newcommand{\supp}{\mathop\textup{supp}}
\renewcommand{\Re}{\mathop\textup{Re}}
\renewcommand{\Im}{\mathop\textup{Im}}

\title[Dual linear programming bounds for sphere packing\dots]{Dual linear
programming bounds\\ for sphere packing via modular forms}

\author{Henry Cohn}
\address{Microsoft Research New England\\
One Memorial Drive\\
Cambridge, MA 02142} \email{cohn@microsoft.com}

\author{Nicholas Triantafillou}
\address{Department of Mathematics\\
Massachusetts Institute of Technology\\
Cambridge, MA 02139}
\curraddr{Department of Mathematics\\
	University of Georgia\\
Athens, GA 30602}
\email{nicholas.triantafillou@gmail.com}

\thanks{Triantafillou was partially supported by an internship at Microsoft Research New
England, a National Science Foundation Graduate Research Fellowship under
grant \#1122374, and Simons Foundation grant \#550033.}

\begin{document}

\begin{abstract}
We obtain new restrictions on the linear programming bound for sphere
packing, by optimizing over spaces of modular forms to produce feasible
points in the dual linear program. In contrast to the situation in
dimensions $8$ and $24$, where the linear programming bound is sharp, we
show that it comes nowhere near the best packing densities known in
dimensions $12$, $16$, $20$, $28$, and $32$.  More generally, we provide a
systematic technique for proving separations of this sort.
\end{abstract}

\maketitle

\section{Introduction}

The sphere packing problem asks for the densest packing of congruent spheres
in $\R^d$.  In other words, what is the greatest proportion of $\R^d$ that
can be covered by congruent balls with disjoint interiors? The case $d=1$ is
trivial, $d=2$ was solved by Thue \cite{T1892}, and $d=3$ was solved by Hales
\cite{H2005} with a computer-assisted proof that has since been formally
verified \cite{Hplus2017}.  These proofs make essential use of the geometry
of packings in $\R^d$ in a way that seems difficult to extend to higher
dimensions, and so another approach is needed when $d$ is large. Based on a
long history of linear programming bounds in coding theory, Cohn and Elkies
\cite{CE2003} developed a linear programming bound for sphere packing.  It
yields the best upper bounds known for the packing density in high dimensions
\cite{CZ2014}, and Cohn and Elkies conjectured that the linear programming
bound is sharp when $d=8$ or $d=24$.

In a recent breakthrough, Viazovska \cite{V2017} proved this conjecture for
$d=8$, and thus showed that the $E_8$ root lattice yields the densest sphere
packing in $\R^8$. Shortly thereafter, Cohn, Kumar, Miller, Radchenko, and
Viazovska \cite{CKMRV2017} proved the conjecture for $d=24$. These are the
only two cases beyond $d=3$ in which the sphere packing problem has been
solved.

These advances raise numerous questions.  Is it possible that the linear
programming bound is sharp in some other dimensions?  Could it even be sharp
in every dimension?  (Surely not, but why not?)  What happens in $\R^{16}$,
and why does that case seemingly not behave like $\R^8$ and $\R^{24}$? These
questions remain mysterious, but in this paper we take some initial steps
towards answering them.

The difficulty in analyzing the linear programming bound stems from the use
of an auxiliary function, which must satisfy certain inequalities. The
quality of the bound depends on the choice of this function, and optimizing
the bound amounts to optimizing a functional over the infinite-dimensional
space of auxiliary functions.  This optimization problem has not been solved
exactly except when $d \in \{1,8,24\}$.  In other dimensions, we can
approximate the true optimum by using a computer to optimize over a
finite-dimensional subspace.  The resulting auxiliary function always proves
some bound for the sphere packing density, and we expect it to be close to
the optimal linear programming bound if the subspace is large and generic
enough.  However, nobody has been able to determine how close it must be.
What if these numerical computations are woefully far from the true optimum?
If that were the case, then they would shed very little light on the linear
programming bound.  It is even possible, albeit implausible, that the linear
programming bound might be sharp for relatively small values of $d$ that
nobody has noticed yet.

\begin{figure}
\begin{center}
\begin{tikzpicture}
\begin{axis}[every x tick/.style={black}, every y tick/.style={black},
ymode = log,
xmin=0, xmax=32, ymin=0.00001,
xlabel=\textup{Dimension},
ylabel=\textup{Sphere packing density},
height=8cm, width=12cm,
domain = 0:32,
enlargelimits = false,
xtick = {4,8,12,16,20,24,28,32},
legend style={at={(0.3,0.075)},anchor=south}]

\addplot +[mark=none, color=black] plot coordinates {
		(1, 1.000000000 )
		(2, 0.906899683 )
		(3, 0.779746762 )
		(4, 0.647704966 )
		(5, 0.524980022 )
		(6, 0.417673416 )
		(7, 0.327455611  )
		(8, 0.253669508 )
		(9, 0.194555339  )
		(10, 0.147953479 )
		(11, 0.111690766  )
		(12, 0.083775831 )
		(13, 0.0624817002 )
		(14, 0.0463644893 )
		(15, 0.0342482621 )
		(16, 0.0251941308 )
		(17, 0.0184640904 )
		(18, 0.0134853405 )
		(19, 0.0098179552)
		(20, 0.0071270537 )
		(21, 0.0051596604)
		(22, 0.0037259420 )
		(23, 0.0026842799)
		(24, 0.0019295744 )
		(25, 0.001384190723)
		(26, 0.000991023890)
		(27, 0.000708229796)
		(28, 0.000505254217)
		(29, 0.000359858186)
		(30, 0.000255902875)
		(31, 0.000181708382)
		(32, 0.000128843289)
};

\addplot +[mark=o, mark size=1.25pt, only marks, color=black] plot
coordinates {
		(1, 1.000000000 )
		(2, 0.906899682 )
		(3, 0.740480489 )
		(4, 0.616850275 )
		(5, 0.465257613 )
		(6, 0.372947545 )
		(7, 0.295297873 )
		(8, 0.253669507 )
		(9, 0.145774875 )
		(10, 0.099615782 )
		(11, 0.066238027  )
		(12, 0.049454176 )
		(13, 0.0320142921 )
		(14, 0.0216240960 )
		(15, 0.0168575706 )
		(16, 0.0147081643 )
		(17, 0.0088113191 )
		(18, 0.0061678981 )
		(19, 0.0041208062 )
		(20, 0.0033945814)
		(21, 0.0024658847 )
		(22, 0.0024510340 )
		(23, 0.0019053281)
		(24, 0.0019295743 )
		(25, 0.00067721200977)
		(26, 0.00026922005043)
		(27, 0.00015759439072)
		(28, 0.00010463810492)
		(29, 0.00003414464690)
		(30, 0.00002191535344)
		(31, 0.00001183776518)
		(32, 0.00001104074930)
}; 	

\addplot[mark=*, mark size=1.25pt, only marks] plot coordinates {
		(12, 0.08338213702  )
		(16, 0.02501194070 )
		(20, 0.00673551674  )
		(28, 0.00048047111  )
		(32, 0.00012085887  )
};
	
\addplot[mark=none] plot coordinates {
		(0,1) }; 	

\legend{\ Linear programming bound\\ Densest known packing\\Dual (lower)
bound\\}
\end{axis}
\end{tikzpicture}
\end{center}
\caption{The upper curve is the linear programming bound computed using the
best auxiliary functions currently known, while the white circles are the
densest sphere packings currently known (see \cite[pp.~xix--xx]{CS1999}). Our
new obstructions, drawn as black circles, are lower bounds for the linear
programming bound.  They show that further optimizing the choice of auxiliary
function cannot improve the linear programming bound by much.}
\label{fig:PackingsBoundsObstructions}
\end{figure}
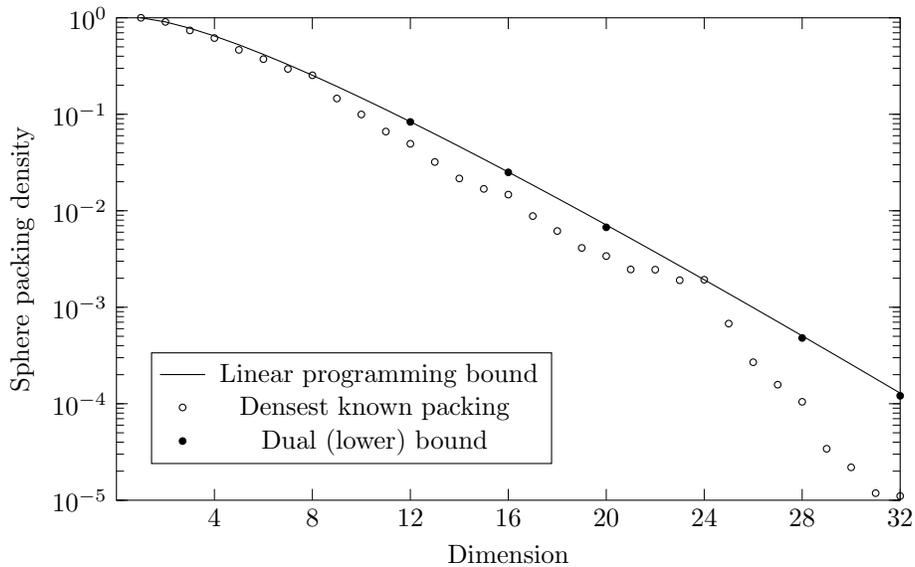

As shown in Figure~\ref{fig:PackingsBoundsObstructions}, the linear
programming bound seems to vary smoothly as a function of dimension, and the
sharp bounds in $8$ and $24$ dimensions fit perfectly with the curve as a
whole. These observations raise our confidence that the numerical
optimization is not in fact misleading. However, there remains a fundamental
gap in the theory of the linear programming bound: how can one prove a
corresponding lower bound, beyond which no auxiliary function can pass?  In
optimization terms, such a bound amounts to a \emph{dual linear programming
bound}, which controls how good the optimal linear programming bound could
be.

In this paper, we show how to compute such a bound when the dimension is a
multiple of four, by optimizing over spaces of modular forms. (We expect that
other dimensions work similarly, but we have not carried out the modular form
calculations in those cases.) Our results for dimensions $12$, $16$, $20$,
$28$, and $32$ are shown in Figure~\ref{fig:PackingsBoundsObstructions} and
Table~\ref{table:results}. The most noteworthy cases are dimensions $12$ and
$16$, where the Coxeter-Todd and Barnes-Wall lattices are widely conjectured
to be optimal sphere packings:

\begin{theorem}
The linear programming bound for the sphere packing density in $\R^{16}$ is
greater than $1.7$ times the density of the Barnes-Wall lattice, and the
bound in $\R^{12}$ is greater than $1.686$ times the density of the
Coxeter-Todd lattice.  In particular, the linear programming bound cannot
prove that either lattice is an optimal sphere packing.
\end{theorem}

Unsurprisingly, in neither case is the linear programming bound even close to
reaching the best density known.  The ratios $1.7$ and $1.686$ are almost
certainly not quite optimal, and we expect that they could be improved to
$1.712$ and $1.694$, respectively, which would match the known upper bounds
to three decimal places.  See Section~\ref{sec:open} for further discussion.

Note that even when the linear programming bound is far from sharp,
determining its value is of interest in its own right.  For example, it can
be interpreted as describing an uncertainty principle for the signs of a
function and its Fourier transform (see \cite{CG2019}). Thus, it has
significance beyond just the topic of sphere packing.

\subsection{The linear programming bound}

Before proceeding further, let us review how the linear programming bound
works.  Recall that a \emph{sphere packing} in $\R^{d}$ is a disjoint union
$\bigcup_{x\in C} B(x,\rho)$ of open unit balls of some fixed radius $\rho$
and centered at the points of some subset $C$ of $\R^{d}$.

Given a sphere packing $\mathcal P$, the \emph{upper density}
$\Delta_{\mathcal P}$ of $\mathcal P$ is defined by
\[
\Delta_{\mathcal P} = \limsup_{r \to \infty} \frac{\vol (B(x,r) \cap \mathcal P)}{\vol(B(x,r))}
\]
for any $x \in \R^d$ (the upper density does not depend on the choice of
$x$).  If the limit exists, and not just the limit superior, then we say that
$\mathcal P$ has \emph{density} $\Delta_{\mathcal P}$.  The \emph{sphere
packing density in $\R^d$} is
\[
\Delta_{d} = \sup_{\mathcal P \subset \R^{d}} \Delta_{\mathcal P},
\]
where the supremum is over sphere packings $\mathcal P$. We will often
renormalize and work with the \emph{upper center density}
\[
\delta_{\mathcal P} = \frac{\Delta_{\mathcal P}}{\vol(B(0,1))} =
\limsup_{r \to \infty} \frac{\# (B(x,r) \cap C)}{\vol(B(x,r))} \cdot \frac{\vol(B(0,\rho))}{\vol(B(0,1))},
\]
which measures the number of center points per unit volume in space if we use
spheres of radius $\rho=1$. Of course the center density has no theoretical
advantage over the density, but it is often convenient not have to carry
around the factor of $\vol(B(0,1)) = \pi^{d/2}/(d/2)!$.  For example,
$\delta_{24}=1$, while $\Delta_{24} = \pi^{12}/12! = 0.00192957\dotsc$.

We normalize the \emph{Fourier transform} $\widehat{f}$ of an integrable
function $f \colon \R^d \to \R$ by
\[
\widehat{f}(y) = \int_{\R^d} f(x) e^{-2\pi i \langle x,y \rangle} \, dx,
\]
where $\langle \cdot,\cdot\rangle$ denotes the usual inner product on $\R^d$.
Cohn and Elkies \cite{CE2003} showed\footnote{Strictly speaking, the paper
\cite{CE2003} imposed stronger hypotheses on $f$, but one can easily remove
those hypotheses by mollifying $f$, using the approach from the first
paragraph of Section~4 in \cite{C2002}.  The fact that they could be removed
was first observed in \cite{CK2007}.} how to use harmonic analysis to bound
the sphere packing density as follows:

\begin{theorem}[Cohn and Elkies \cite{CE2003}]\label{thm:LPbd}
Let $f \colon \R^d \to \R$ be a continuous, integrable function, such that
$\widehat{f}$ is integrable as well and $\widehat{f}$ is real-valued (i.e.,
$f$ is even). Suppose $f$ and $\widehat{f}$ satisfy the following
inequalities for some positive real number $r$:
\begin{enumerate}
\item $f(0)>0$ and $\widehat{f}(0) > 0$,

\item $f(x) \le 0$ for $|x| \ge r$, and

\item $\widehat{f}(y) \ge 0$ for all $y$.
\end{enumerate}
Then every sphere packing in $\R^d$ has upper center density at most
\[
\frac{f(0)}{\widehat{f}(0)} \cdot \left(\frac{r}{2}\right)^d.
\]
\end{theorem}

The \emph{linear programming bound} in $\R^d$ is the infimum of the center
density upper bound
\[
\frac{f(0)}{\widehat{f}(0)} \cdot \left(\frac{r}{2}\right)^d
\]
over all \emph{auxiliary functions} $f$ satisfying the hypotheses of
Theorem~\ref{thm:LPbd}.  See Figure~\ref{fig:sample} for an example of an
auxiliary function, which is far from optimal.

\pgfplotsset{ymin=-0.4, ymax=1.2}
\begin{figure}
\begin{center}
\begin{tikzpicture}
\begin{axis}[every x tick/.style={black}, every y tick/.style={black},no markers,domain=-2:2,
samples=256,axis lines*=left, y label style={rotate=-90}, xlabel=$x$,
ylabel=$\displaystyle f(x)$,  height=3cm, width=10cm,
enlargelimits=false, clip=false, axis on top,
xtick = {-2,-1,0,1,2}, ytick = {0,1}]
\draw [gray]  (axis cs:-2,0) -- (axis cs:2,0);
\addplot[color=black] {lp_f(x)};
\end{axis}
\end{tikzpicture}
\vskip 0.5cm
\begin{tikzpicture}
\begin{axis}[every x tick/.style={black}, every y tick/.style={black},no markers,domain=-2:2,
samples=256,axis lines*=left, y label style={rotate=-90}, xlabel=$y$,
ylabel=$\displaystyle \widehat{f}(y)$,  height=3cm, width=10cm,
enlargelimits=false, clip=false, axis on top,
xtick = {-2,-1,0,1,2}, ytick = {0,1}]
\draw [gray]  (axis cs:-2,0) -- (axis cs:2,0);
\addplot[color=black] {lp_fhat(x)};
\end{axis}
\end{tikzpicture}
\caption{A sample auxiliary function and its Fourier transform
(namely, $f(x) = (1-x^2) e^{-x^2}$ on
$\R^1$, with $r=1$).} \label{fig:sample}
\end{center}
\end{figure}
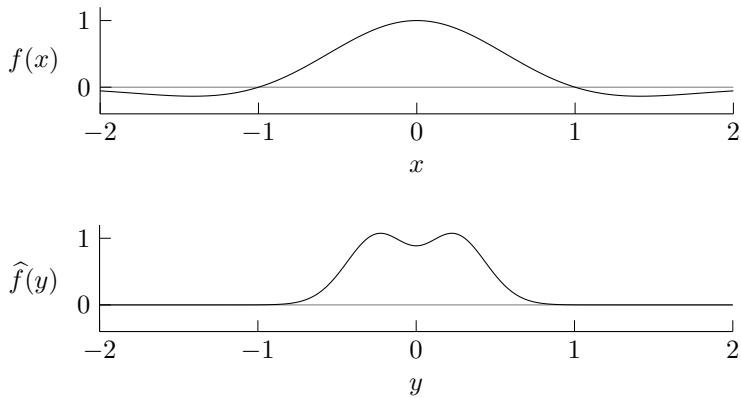

Without loss of generality, we can assume that the auxiliary function $f$ is
radial, because we can simply average its rotations about the origin.  For a
radial function $f$, we write $f(t)$ with $t \in [0,\infty)$ to denote the
common value $f(x)$ with $|x|=t$.  If $f$ is radial, then $\widehat{f}$ is
radial as well, and
\[
\widehat{f}(y) = \frac{2\pi}{|y|^{d/2-1}} \int_{0}^{\infty}  f(t) J_{d/2-1}(2\pi t|y|) t^{d/2} \,dt,
\]
where $J_{d/2-1}$ is the Bessel function of the first kind of order $d/2-1$
(see, for example, Theorem~9.10.3 in \cite{AAR1999}).

The density bound
\[
\frac{f(0)}{\widehat{f}(0)} \cdot \left(\frac{r}{2}\right)^d
\]
is invariant under replacing $f$ with $x \mapsto f(\rho x)$ and $r$ with
$r/\rho$ for any scaling factor $\rho \in (0,\infty)$.  Without loss of
generality we can use this invariance to fix $r=1$, and we can assume
$\widehat{f}(0)=1$ as well. Then the constraints on $f$ from
Theorem~\ref{thm:LPbd} are linear inequalities, and the density bound is also
a linear functional of $f$. Thus, optimizing the choice of $f$ amounts to
solving an infinite-dimensional linear optimization problem, which explains
the name ``linear programming bound.'' In practice, however, fixing $r$ may
not lead to the prettiest answers.  For example, Cohn and Elkies found more
elegant behavior if one instead fixes $f(0)=\widehat{f}(0)$ and lets $r$ vary
(see Section~7 of \cite{CE2003}).

The best choice of $f$ is not known, except when $d \in \{1,8,24\}$, and
little is known about how good the optimal bound might be.  It is not hard to
produce upper bounds by numerically optimizing over finite-dimensional spaces
of functions, and in most cases these upper bounds seem to be close to the
optimal linear programming bound (see \cite{ACHLT} for the most extensive
calculations so far). However, these computational methods leave open the
possibility that other auxiliary functions might prove much better bounds.

What sort of obstructions prevent the linear programming bound from reaching
the density of the best sphere packing?  In this paper we provide a partial
answer, with an algorithm to compute such obstructions via linear programming
over spaces of modular forms of weight $d/2$. The algorithm is based on
optimizing a summation formula for radial Schwartz functions, which is an
analogue of Voronoi summation.

The remainder of the paper is organized as follows. In
Section~\ref{generalSetup}, we present a general framework for computing dual
linear programming bounds.  We describe our algorithm in
Section~\ref{algorithm}, and we prove the summation formula underlying the
algorithm in Section~\ref{poisson}.  In Section~\ref{positivityCheck}, we
expand on the final step of our algorithm by describing a method for checking
in finite time that all of the coefficients of the $q$-expansion of a given
modular form are nonnegative. Finally, we present a table of new lower bounds
in Section~\ref{results}, and we conclude with open problems in
Section~\ref{sec:open}.

\section{Duality} \label{generalSetup}

Computing a bound for the objective function in a linear program is typically
straightforward: it just amounts to finding a feasible point in the dual
linear program.  The difficulty in our case is that the optimization problems
are infinite-dimensional.  The primal problem is relatively tractable,
because the auxiliary functions in Theorem~\ref{thm:LPbd} are well behaved in
practice.  We can approximate them with polynomials times Gaussians, and
using high-degree polynomials yields excellent results.  For example, in
$\R^{16}$ the resulting center density bounds seem to converge to
\[
0.10705844234092448845891681517141\dotsc
\]
as the polynomial degree tends to infinity, and we believe this number is the
optimal linear programming bound for $16$ dimensions, correct to $32$ decimal
places. Unfortunately, the dual problem is much less tractable.  It amounts
to optimizing over a space of measures, and we believe the optimal measures
will be singular (specifically, supported on a discrete set of radii). In
particular, we know of no simple family of measures we can use to approximate
them fruitfully. Instead, the dual problem appears to be quite a bit more
subtle.

In Section~4 of \cite{C2002}, Cohn formulated the dual linear program as
follows. Here, $\delta_0$ denotes a delta function at the origin, and
$\widehat{\mu}$ is the Fourier transform of $\mu$ as a tempered distribution.

\begin{proposition} \label{prop:dualLP}
Let $\mu$ be a tempered distribution on $\R^d$ such that $\mu = \delta_0 +
\nu$ with $\nu \ge 0$, $\supp(\nu) \subseteq \{x \in \R^d : |x| \ge r\}$ for
some $r>0$, and $\widehat{\mu} \ge c \delta_0$ for some $c>0$. Then the
linear programming bound in $\R^d$ is at least
\[
c \cdot \left(\frac{r}{2}\right)^d.
\]
\end{proposition}

\begin{proof}[Sketch of proof]
Let $f\colon \R^d \to \R$ be an auxiliary function satisfying the hypotheses
of Theorem~\ref{thm:LPbd}, where we use scaling invariance to ensure that the
same value of $r$ works for both $f$ and $\mu$.  If $f$ and $\widehat{f}$ are
rapidly decreasing, then the inequalities on $f$ and $\mu$ imply that
\[
f(0) \ge \int_{\R^d} f \mu = \int_{\R^d} \widehat{f} \,\widehat{\mu} \ge c \widehat{f}(0),
\]
and thus
\[
\frac{f(0)}{\widehat{f}(0)} \ge c,
\]
as desired.  More general auxiliary functions must be mollified, as described
in Section~4 of \cite{C2002}, after which the same argument applies to them
as well.
\end{proof}

The difficulty in applying this proposition is how to find a plentiful source
of distributions $\mu$ that could satisfy the hypotheses.  One source is
Poisson summation for lattices, which says that for any lattice $\Lambda$ in
$\R^d$, the Fourier transform of the distribution
\[
\sum_{x \in \Lambda} \delta_x
\]
is
\[
\frac{1}{\vol(\R^d/\Lambda)} \sum_{y \in \Lambda^*} \delta_y,
\]
where $\Lambda^*$ is the dual lattice. Thus, the hypotheses of
Proposition~\ref{prop:dualLP} are satisfied with $c= 1/\vol(\R^d/\Lambda)$
and $r = \min_{x \in \Lambda \setminus\{0\}} |x|$. The resulting lower bound
amounts to proving Theorem~\ref{thm:LPbd} for lattice packings.

In principle, one could try to improve on individual lattices by using a
linear combination of Poisson summation formulas for different lattices (see,
for example, the bottom of page~351 in \cite{C2002}).  However, that does not
seem fruitful in general. Instead, we use the following analogue of Voronoi
summation to produce distributions from modular forms. For definitions
related to modular forms, see \cite{DS2005}. In
particular, recall that the slash operator is defined as follows: if $M = \begin{pmatrix} a & b\\
c & d\end{pmatrix} \in \GL_2(\R)$ and $\det M > 0$, then
\[
(f|_k M)(z) := (ad-bc)^{k/2} (cz+d)^{-k} f\mathopen{}\left(\frac{az+b}{cz+d}\right)\mathclose{}.
\]

\begin{proposition} \label{PoissonAnalogue}
Let $d  = 2k$ with $k \in \N$, let $g\in M_{k}(\Gamma_{1}(N))$ be a modular
form of weight $k$ for the congruence subgroup $\Gamma_{1}(N)$, let $w_{N} = \begin{pmatrix} 0 & -1 \\
N & 0
\end{pmatrix}$, and let
\[
 \widetilde{g}(z) = i^k (g|_k{w_{N}})(z) = \frac{i^k}{N^{k/2}z^{k}} g\mathopen{}\left(-\frac{1}{Nz}\right)\mathclose{}
\]
be $i^k$ times the image of $g$ under the full level $N$ Atkin-Lehner
operator (so that $g = i^k \widetilde{g}|_k{w_{N}}$ as well).  Let the
$q$-expansions of $g$ and $\widetilde{g}$ be
\[
 g(z) = \sum_{n=0}^{\infty} a_{n}q^{n} \qquad\text{and}\qquad \widetilde{g}(z) = \sum_{n=0}^{\infty} b_{n} q^{n},
\]
where $q = e^{2\pi i z}$. Then for every radial Schwartz function $f \colon
\R^d \to \C$,
\[
\sum_{n=0}^{\infty} a_{n} f(\sqrt{n}) =
\left(\frac{2}{\sqrt{N}}\right)^{d/2}
\sum_{n=0}^{\infty} b_{n} \widehat{f}\mathopen{}\left(\frac{2\sqrt{n}}{\sqrt{N}}\right)\mathclose{}.
\]
\end{proposition}

In particular, if $\delta_r$ denotes a delta function supported on the sphere
of radius $r$ about the origin in $\R^d$, then this proposition says that the
tempered distributions
\[
\sum_{n=0}^\infty a_n \delta_{\sqrt{n}}
\qquad \text{and}\qquad
\left(\frac{2}{\sqrt{N}}\right)^{d/2}
\sum_{n=0}^{\infty} b_{n} \delta_{2\sqrt{n/N}}
\]
are Fourier transforms of each other. Our algorithm will optimize over
distributions of this form. The advantage of these distributions is that
their supports help enforce the constraint that $\supp(\nu) \subseteq \{x \in
\R^d : |x| \ge r\}$ in Proposition~\ref{prop:dualLP}.

For comparison, the techniques in Section~5 of \cite{CG2019} produce what
appear to be close numerical approximations to the optimal distributions
$\mu$. They have the form
\[
\mu = \sum_{n \ge 0} c_n \delta_{r_n}
\]
with radii given by $0 = r_0 < r_1 < r_2 < \dotsb$ and tending to infinity,
coefficients $c_n>0$, and $\widetilde{\mu} = \mu$. For example, in $\R^{16}$
the first few radii and coefficients are listed in
Table~\ref{table:dual-distribution}. The only drawback is that the results of
these calculations are merely conjectural: we do not know whether such a
distribution actually exists.

Our approach in this paper amounts to approximating the optimal $\mu$ with a
distribution $\mu'$ whose existence follows from
Proposition~\ref{PoissonAnalogue}.  For comparison,
Table~\ref{table:dual-distribution} shows the best $\mu'$ we have obtained,
which we computed using the parameters $N=96$ and $T=20$ in the notation of
the next section. This distribution is of the form $\mu' = \sum_{n \ge 0}
c'_n \delta_{r'_n}$, with Fourier transform $\widehat{\mu}' = \sum_{n \ge 0}
c''_n \delta_{r''_n}$. In the table, we have rescaled the distribution $\mu'$
so that $c'_0 = c''_0 = 1$.  Note that
\begin{align*}
r_1 &\approx r'_1 \approx r''_1 \approx r''_2,\\
r_2 &\approx r'_2 \approx r'_3 \approx r''_3 \approx r''_4, \text{ and}\\
r_3 &\approx r'_4 \approx r'_5 \approx r''_5 \approx r''_6,
\end{align*}
and the sums of the corresponding coefficients are also near each other.  The
approximation to $\mu$ is not yet very close, but one can already see $\mu$
roughly emerging from $\mu'$.

\begin{table}
\caption{Radii and coefficients for dual distributions in $\R^{16}$.}
\label{table:dual-distribution}
\begin{tabular}{d{1.3}d{1.19}d{8.11}}
\toprule
n & r_n & c_n\\
0 & 0 & 1\\
\rowcolor{Gray}
1 & 1.7393272583625204\dotsc   &     8431.71627140\dotsc\\
2 & 2.2346642069957498\dotsc   &   292026.09352080\dotsc\\
\rowcolor{Gray}
3 & 2.6462005756471079\dotsc   &  3111809.14450639\dotsc\\
\midrule
n & r'_n & c'_n\\
0 & 0 & 1\\
\rowcolor{Gray}
1 & 1.7385384653461733\dotsc &       8360.61230142\dotsc\\
2 & 2.1990965401230488\dotsc &       4240.44226222\dotsc\\
3 & 2.2331930934327142\dotsc &     282582.90774253\dotsc\\
\rowcolor{Gray}
4 & 2.6366241274825130\dotsc &    2419678.28385080\dotsc\\
\rowcolor{Gray}
5 & 2.6651290005171109\dotsc &     584982.54962505\dotsc\\
\midrule
n & r''_n & c''_n\\
0 & 0 & 1\\
\rowcolor{Gray}
1 & 1.6604472109700065\dotsc &        133.02471778\dotsc\\
\rowcolor{Gray}
2 & 1.7414917267847931\dotsc &       8321.61159562\dotsc\\
3 & 2.2277237020673214\dotsc &     245869.54859549\dotsc\\
4 & 2.2887685306282807\dotsc &      50042.27252495\dotsc\\
\rowcolor{Gray}
5 & 2.6253975605696717\dotsc &    1578408.61282183\dotsc\\
\rowcolor{Gray}
6 & 2.6773906784567302\dotsc &    1610965.69273527\dotsc\\
\bottomrule
\end{tabular}
\end{table}

It is natural to ask why $\Gamma_1(N)$ appears in Proposition~\ref{PoissonAnalogue},
rather than the group $\Gamma_0(N)$ that is more often used in the theory of modular forms.
This is simply a matter of generality: $\Gamma_1(N)$ is a subgroup of $\Gamma_0(N)$, and thus
$M_{k}(\Gamma_{1}(N))$ contains $M_{k}(\Gamma_{0}(N))$ and is generally larger. For comparison,
$\Gamma(N)$ is an even smaller group, but it is not closed under conjugation by $w_N$, which
means the map $g \mapsto \widetilde{g}$ used in Proposition~\ref{PoissonAnalogue} 
does not preserve $M_{k}(\Gamma(N))$.

\section{An algorithm for dual linear programming bounds} \label{algorithm}

Instead of using modular forms for the congruence
subgroup $\Gamma_1(N)$, for simplicity we will restrict our attention to
those for the larger group $\Gamma_0(N)$ (equivalently, to modular forms for
$\Gamma_1(N)$ that have trivial Nebentypus).  This restriction entails some loss of
generality: for example, $\Gamma_0(N)$ does not have modular forms of odd weight,
while $\Gamma_1(N)$ often does. However, $\Gamma_0(N)$ serves as an
attractive proving ground for the
general theory, and it should suffice when the dimension $d$ is a multiple of
$4$.

Specifically, let $k=d/2$ be an even integer, and let $M_k(\Gamma_0(N))$ be
the space of modular forms of weight $k$ for $\Gamma_0(N)$.  Recall that this
space has a basis consisting of modular forms with rational coefficients in
their $q$-expansions (see, for example, Corollary~12.3.12 in \cite{DI1995}).
Furthermore, the Atkin-Lehner involution on $M_k(\Gamma_0(N))$ preserves the
property of having rational coefficients (see Lemma~3.5.3 in \cite{O1995}).

In practice, to simplify Section~\ref{positivityCheck} we also assume that
$N$ is not divisible by $16^2$, $9^2$, or $p^2$ for any prime $p > 3$, but
this assumption is not essential.

We would like to find a modular form $g =\sum_{n \ge 0} a_n q^n$ in
$M_k(\Gamma_0(N))$ with the following properties for some $T$, where we set
$\widetilde{g} = i^k g|_k w_N = \sum_{n \ge 0} b_n q^n$:
\begin{enumerate}
\item $a_0=1$ and $b_0>0$,

\item $a_n \ge 0$ and $b_n \ge 0$ for all $n \ge 0$, and

\item $a_n = 0$ for $1 \le n < T$.
\end{enumerate}
Then we use the distribution
\[
\mu = \sum_{n \ge 0} a_n \delta_{\sqrt{n}}
\]
in Proposition~\ref{prop:dualLP}.  By Proposition~\ref{PoissonAnalogue}, we
have $c=(2/\sqrt{N})^{d/2}b_0$ and $r = \sqrt{T}$ in the notation of
Proposition~\ref{prop:dualLP}.  Thus, we obtain a lower bound of
\[
b_0 \left(\frac{2}{\sqrt{N}}\right)^{d/2} \left(\frac{\sqrt{T}}{2}\right)^d
\]
for the linear programming bound in $\R^d$, and we wish to choose $g$ so as
to maximize this bound.  We will do so by linear programming, with one
caveat: all our calculations will consider only the terms up to $q^M$ in the
$q$-series for some fixed $M$, and at the end we must check that the
inequalities are not violated beyond that point.

Let $g^1, \dots, g^{\dim M_k(\Gamma_0(N))}$ be a basis of $M_k(\Gamma_0(N))$
with rational $q$-series coefficients, and let $\widetilde{g}^{j} = i^k
g^{j}|_k{w_{N}}$ be $i^k$ times the image of $g^{j}$ under the full level $N$
Atkin-Lehner involution. We write the $q$-expansions of the modular forms
$g^j$ and $\widetilde{g}^j$ as
\[
g^{j} = \sum_{n = 0}^{\infty} a^{j}_{n} q^{n} \qquad\text{and}\qquad \widetilde{g}^{j} = \sum_{n=0}^{\infty} b^{j}_{n}q^{n},
\]
and we fix integers $T$ and $M$ with $1 \leq T < \dim M_k(\Gamma_0(N)) < M$.
These bases and $q$-series can all be computed algorithmically (see, for
example, \cite{S2007}).

Now we write $g = \sum_j x_j g^j$ with respect to our basis, and we optimize
over the choice of coefficients $x_j$ by solving the following linear
program:
\[
\begin{array}{ll}
\text{maximize} & \sum_j x_{j} b^{j}_{0} \\
\text{subject to} & 1 = \sum_j x_{j} a^{j}_{0}, \\
& 0 = \sum_j x_{j} a^{j}_{n} \text{ for $1 \leq n < T$,}\\
& 0 \leq \sum_j x_{j} a^{j}_{n} \text{ for $T \leq n \leq M$, and}\\
& 0 \leq \sum_j x_{j} b^{j}_{n} \text{ for $1 \leq n \leq M$.}
\end{array}
\]
These inequalities encode all the desired properties of $f$ and $g$, except
that we examine only the terms up to $q^M$ in the $q$-series.

We hope that if $M$ is large enough, then all the terms beyond $q^M$ will
have nonnegative coefficients automatically, and we attempt to use asymptotic
bounds to confirm that all of the coefficients of $g$ and $\widetilde{g}$ are
nonnegative (see Section~\ref{positivityCheck}). If this verification fails,
we can increase $M$ and attempt the optimization problem again. In practice,
$M = 2 \cdot \dim M_k(\Gamma_0(N))$ typically seems to be sufficient for the
algorithm to succeed, and it works for all the numerical results we report in
this paper.

To find the best possible bounds, we run the method for several values of $N$
and $T$. Larger values of $N$ typically yield better results, but not always.
It seems difficult to predict the best values for $T$ in general, although
they also tend to increase as $N$ increases. See Section~\ref{results} for
the results of this method applied to the spaces $M_k(\Gamma_0(N))$ of
modular forms of weight $k \in \{6,8,10,14,16\}$ and level $N = 24$ or $96$.

For a concrete illustration of the method, consider the case $d=16$ and
$N=4$. One can show that the space $M_8(\Gamma_0(4))$ is five-dimensional,
with the following basis. Let
\[
E_8(z) = 1 + 480\sum_{n=1}^\infty \sigma_7(n) q^n
\]
be the Eisenstein series of weight $8$ for $\SL_2(\Z)$ (not to be confused
with the $E_8$ root lattice), and let $f$ be the newform of weight $8$ for
$\Gamma_0(2)$ defined by
\[
f(z) = q \prod_{n=1}^\infty (1-q^n)^8 (1-q^{2n})^8.
\]
Then $M_g(\Gamma_0(4))$ has the basis $g^1,\dots,g^5$, where $g^1(z) =
E_8(z)$, $g^2(z) = 16 E_8(2z)$, $g^3(z) = 256 E_8(4z)$, $g^4(z) = f(z)$, and
$g^5(z) = 16 f(2z)$.  The Atkin-Lehner involution acts by $\widetilde{g}^1 =
g^3$, $\widetilde{g}^2 = g^2$, $\widetilde{g}^3 = g^1$, $\widetilde{g}^4 =
g^5$, and $\widetilde{g}^5 = g^4$. Using this information, we can write down
the linear program explicitly and solve it. As usual, the trickiest part is
identifying the right choice of $T$, while we can simply take $M$ large
enough (e.g., $M=10$ is more than sufficient).

For $T=2$, solving the linear program yields the modular form
\begin{align*}
\sum\nolimits_j x_j g^j &= \frac{1}{17} g^1 + \frac{1}{17} g^2 - \frac{480}{17} g^4\\
&= 1+4320q^2 + 61440q^3 + 522720q^4 + \dotsb,
\end{align*}
which is the theta series of the Barnes-Wall lattice.  Similarly, for $T=4$
we obtain
\begin{align*}
\sum\nolimits_j x_j g^j &= \frac{1}{272} g^2 + \frac{1}{272} g^3 - \frac{30}{17} g^5\\
&= 1+4320q^4 + 61440q^6 + 522720q^8 + \dotsb,
\end{align*}
which is the same modular form with $q$ replaced by $q^2$ and which yields
the same bound.  For these two values of $T$, the space $M_8(\Gamma_0(4))$ is
incapable of separating the linear programming bound from the center density
$0.0625$ of the Barnes-Wall lattice. However, for $T=3$ we obtain
\begin{align*}
\sum\nolimits_j x_j g^j &= \frac{1}{136} g^1  -\frac{121}{2176}g^2 + \frac{1}{136}g^3  -\frac{60}{17} g^4  -\frac{60}{17} g^5\\
& = 1 + 7680q^3 + 4320q^4 + 276480q^5 + \dotsb,
\end{align*}
which yields an improved center density lower bound of $3^8/2^{16} =
0.100112\dotsc$, more than $60\%$ greater than the center density of the
Barnes-Wall lattice. In fact, this modular form has been studied before: it
is the \emph{extremal theta series} in $16$ dimensions (see equation (47) in
\cite[p.~190]{CS1999}).

It is tempting to conjecture that the extremal theta series should exactly
match the optimal linear programming bound.  This conjecture would be a
beautiful analogue of the behavior in $8$ and $24$ dimensions.  In those
cases the optimal lattices have determinant $1$ and minimal norm $2$ or $4$,
respectively.  The extremal theta series in $16$ dimensions behaves like the
theta series of a lattice of determinant $1$ and minimal norm $3$, exactly
interpolating between $8$ and $24$ dimensions.  No such lattice exists
\cite{MV2010}, but the linear programming bound could match the density of a
hypothetical lattice.

That is a good approximation in this case, but the answer turns out to be
more subtle: in Section~\ref{results}, we obtain a better lower bound using
$N=96$.  Instead of minimal norm $3$, the improved lower bound is $3.022$.
For comparison, we believe the true linear programming bound amounts to a
minimal norm of \[3.02525931168288206328208655790196\dotsc,\] but we are
unable to conjecture an exact formula for this number.

\section{Poisson summation analogues from modular forms} \label{poisson}

The main result of this section is Proposition~\ref{PoissonAnalogue}, which
yields a summation formula from a modular form. Summation formulas of this
sort are well known to number theorists, and essentially equivalent to the
functional equation for the $L$-function. We record the details here and
sketch a proof for the convenience of the reader. (One can also prove such a
formula using the density of complex Gaussians among radial Schwartz
functions, along the lines of Section~6 in \cite{RV2018} or Section~2.3 in
\cite{CKMRV2019}.)

Proposition~\ref{PoissonAnalogue} is essentially a version of Voronoi
summation. Our proof will follow the approach used in standard proofs of
Voronoi summation (for example, as in Section~10.2.5 of \cite{C2007} or
Section~2 of \cite{MS2004}). The key idea comes from the classical
observation that the usual Poisson summation formula is a consequence of the
functional equation of the Riemann zeta function. Similarly,
Proposition~\ref{PoissonAnalogue} follows from the functional equation
relating the $L$-functions associated to a modular form and its Atkin-Lehner
dual.

In what follows, we use the notation established in
Proposition~\ref{PoissonAnalogue}. To state the functional equation, we first
define the \emph{$L$-function}
\[
L(s,g) = \sum_{n=1}^{\infty} \frac{a_{n}}{n^{s}}
\]
when $\Re(s)>k$, and the \emph{completed $L$-function}
\[
\Lambda(s,g) = N^{s/2}(2\pi)^{-s} \Gamma(s) L(s,g).
\]
The functional equation relating $\Lambda(s,g)$ and
$\Lambda(s,\widetilde{g})$ is classical, dating back to Hecke \cite{H1936}.
It says that the $L$-functions can be analytically continued so that
\[
\Lambda(s,g) + \frac{a_0}{s} + \frac{b_0}{k-s}
\]
is entire and bounded in every vertical strip, and we have the functional
equation
\[
\Lambda(s,g) = \Lambda(k-s, \widetilde{g}),
\]
or equivalently
\begin{equation} \label{eq:funceq}
L\mathopen{}\left(k - \frac{s}{2}, g\right)\mathclose{} = N^{(s-k)/2} (2\pi)^{k-s} \frac{\Gamma(s/2)}{\Gamma(k - s/2)}
L\mathopen{}\left(\frac{s}{2}, \tilde{g}\right)\mathclose{}.
\end{equation}
See, for example, Theorem~1 in \cite[p.~I-5]{O1969}.

\begin{proof}[Sketch of proof of Proposition~\ref{PoissonAnalogue}]
For a radial Schwartz function $f$ on $\R^d$, let
\[
S= \sum_{n\geq 1} a_n
f(\sqrt{n}).
\]
By Mellin inversion,
\[
a_n f(\sqrt{n}) = \frac{1}{2\pi i} \int_{\Re(s) = \sigma} \frac{a_{n}}{n^{s/2}} \M f(s) \,ds
\]
for any $\sigma > 0$, where the \emph{Mellin transform} $\M f$ is defined by
\[
\M f (s) = \int_0^\infty f(x) x^s \frac{dx}{x}.
\]
In particular, for $\sigma = d + \varepsilon$ with $\varepsilon>0$,
\begin{align*}
S   & = \frac{1}{2\pi i} \sum_{n=1}^{\infty} \int_{\Re(s) = d+\varepsilon} \frac{a_{n}}{n^{s/2}} \M f(s) \,ds \\
& = \frac{1}{2\pi i} \int_{\Re(s) = d + \varepsilon} L\mathopen{}\left(\frac{s}{2}, g\right)\mathclose{} \M f(s) \,ds,
\end{align*}
where switching the sum and integral is permitted because of the uniform
convergence of the sum defining the $L$-function.

The integrand $L(s/2, g) \M f(s)$ is negligible when $s$ has large imaginary
part.  To see why, note that by a stationary phase argument the Mellin
transform $\M f(s)$ is rapidly decaying as $\Im(s)$ grows, while $L(s/2, g)$
grows at most polynomially in $\Im(s)$ by the Phragm\'en-Lindel\"of
principle.  Thus, we can shift the contour of integration to the left, as
long as we account for poles.

It is not hard to check that $\M f(s)$ has a possible pole at $s=0$ with
residue $f(0)$, $L(s/2, g)$ has a possible pole at $s=d$ with residue \[ 2
\left(\frac{2\pi}{\sqrt{N}}\right)^{d/2} \frac{1}{\Gamma(d/2)} b_0,
\]
and $L(0,g) = -a_{0}$, since the pole of $\Gamma(s)$ at $s=0$ cancels the
pole of $\Lambda(s,g)$ at $s=0$. Thus,
\[
S =
-a_0 f(0) + 2 b_0 \left(\frac{2 \pi}{\sqrt{N}}\right)^{d/2}
\frac{1}{\Gamma(d/2)} \M f(d) + \frac{1}{2\pi i}  \int_{\Re(s)
=-\varepsilon} L\mathopen{}\left(\frac{s}{2}, g\right)\mathclose{} \M f(s) \,ds.
\]

Setting
\[
 T = \frac{1}{2\pi i}  \int_{\Re(s) =-\varepsilon} L\mathopen{}\left(\frac{s}{2}, g\right)\mathclose{} \M f(s) \,ds
\]
and applying the identity $\widehat{f}(0) = \frac{2\pi^{d/2}}{\Gamma(d/2)}\M
f(d)$, we see that
\begin{equation} \label{eq:include0}
a_0 f(0) + S = \left(\frac{2}{\sqrt{N}}\right)^{d/2} b_{0} \widehat{f}(0) + T.
\end{equation}
Changing variables from $s$ to $d-s$ and applying the functional equation
\eqref{eq:funceq} yields
\begin{align*}
 T & = \frac{1}{2\pi i} \int_{\Re(s) = d + \varepsilon} L\mathopen{}\left(\frac{d-s}{2}, g\right)\mathclose{} \M f(d-s) \,ds \\
& = \frac{1}{2\pi i} \int_{\Re(s) = d+\varepsilon} N^{s/2-d/4} (2\pi)^{d/2-s} \frac{\Gamma(s/2)}{\Gamma((d-s)/2)}
L\mathopen{}\left(\frac{s}{2}, \widetilde{g}\right)\mathclose{}\M f(d-s) \,ds.
\end{align*}

Now we use the identity
\[
\M \widehat{f}(s) = \frac{\pi^{d/2 - s} \Gamma(s/2)}{\Gamma((d - s)/2)} \M f(d - s)
\]
(see Theorem~5.9 in \cite{LL2001}). Making this substitution, we find that
\[
 T  = \left(\frac{2}{\sqrt{N}}\right)^{d/2} \frac{1}{2\pi i} \int_{\Re(s) = d+\varepsilon} \left(\frac{4}{N}\right)^{-s/2}
 L\mathopen{}\left(\frac{s}{2}, \widetilde{g}\right)\mathclose{} \M \widehat{f}(s) \,ds.
\]
Replacing the $L$-function with its defining sum, switching the sum and
integral as above, and applying Mellin inversion again (reversing the steps
from the start of the proof), we see that
\begin{align*}
T & = \left(\frac{2}{\sqrt{N}}\right)^{d/2} \frac{1}{2\pi i}  \int_{\Re(s) = d+\varepsilon} \sum_{n=1}^{\infty} \frac{b_{n}}{(4n/N)^{s/2}}\M \widehat{f}(s) \,ds \\
& = \left(\frac{2}{\sqrt{N}}\right)^{d/2} \sum_{n=1}^{\infty} \frac{1}{2\pi i}  \int_{\Re(s) = d+\varepsilon} \frac{b_{n}}{(4n/N)^{s/2}}\M \widehat{f}(s) \,ds \\
& = \left(\frac{2}{\sqrt{N}}\right)^{d/2} \sum_{n=1}^{\infty} b_n \widehat{f}\mathopen{}\left(\frac{2\sqrt{n}}{\sqrt{N}} \right)\mathclose{}.
\end{align*}
Hence, \eqref{eq:include0} implies that
\[
\sum_{n=0}^{\infty} a_{n} f(\sqrt{n}) = \left(\frac{2}{\sqrt{N}}\right)^{d/2} \sum_{n=0}^{\infty} b_{n}
\widehat{f}\mathopen{}\left(\frac{2\sqrt{n}}{\sqrt{N}}\right)\mathclose{},
\]
as desired.
\end{proof}

\section{Checking positivity of modular form coefficients} \label{positivityCheck}

In this section, we explain how we check whether a modular form of weight $k$
for $\Gamma_0(N)$ has nonnegative coefficients in its $q$-series.  This
method uses only standard techniques from the theory of modular forms, but we
describe them here for the benefit of readers in discrete geometry.  The key
idea is that Eisenstein series typically make the dominant contribution
asymptotically, which reduces the problem to a finite calculation if the
Eisenstein contribution is positive.

As mentioned above, we assume for simplicity that $N$ is not divisible by
$16^2$, $9^2$, or $p^2$ for any prime $p > 3$. This assumption guarantees
that all the characters in this section are real. Furthermore, we assume that
$k \ge 3$, because the Eisenstein series for weight $2$ must be obtained
using different formulas (the formulas that work for $k \ge 3$ no longer
converge when $k=2$).

To verify that $g = \sum_{n=0}^{\infty} a_n q^n$ has $a_n \geq 0$ for all
$n$, we write $g$ as $g_e +g_c$, where $g_e = \sum_{n=0}^{\infty} e_n q^n$ is
a linear combination of Eisenstein series and $g_c = \sum_{n=0}^{\infty} c_n
q^n$ is cuspidal, and we attempt to carry out the following steps:
\begin{enumerate}
\item Use Weil bounds to show that $|c_n| \le C_g n^{k/2}$ for some
    explicit constant $C_g$.

\item Use explicit formulas for Eisenstein series to show that $e_n \ge r_g
    n^{k-1}$ for some explicit constant $r_g>0$.

\item Compare the Eisenstein part and the cuspidal part to produce a bound
    $Q$ such that $a_n > 0$ for $n > Q$.

\item Explicitly compute the coefficients $a_n$ of $g$ to check that $a_n
    \geq  0$ for $n \le Q$.
\end{enumerate}

The first step is straightforward, given some powerful machinery. Deligne's
proof of the Weil conjectures \cite{D1974} implies that, independent of
weight, if  $h = \sum_{n=1}^{\infty} c_n q^n$ is a cuspidal Hecke eigenform
normalized so that $c_{n'} = 1$ for the minimal $n'$ with $c_{n'} \neq 0$,
then $|c_n| \leq \sigma_0(n) n^{(k-1)/2} \leq n^{k/2}$. Let $B_k(N)$ be the
set of such eigenforms, which are a basis for the cuspidal part of
$M_k(\Gamma_0(N))$. (Note that the elements of $B_k(N)$ typically do not have
rational coefficients.  Instead, we must work over a larger number field.) If
\[
g_c  = \sum_{n=1}^{\infty} c_n q^n = \sum_{h \in B_k(N)} x_h h
\]
with coefficients $x_h \in \C$, then
\[
|c_n| \leq  n^{k/2} \sum_{h\in B_k(N)}
|x_h|.
\]
Thus, step~(1) holds with $C_g = \sum_{h\in B_k(N)} |x_h|$.

For the second step, we need to write down the Eisenstein series explicitly.
We can describe them in terms of primitive Dirichlet characters $\phi$ of
conductor $u$ and natural numbers $t$ such that $u^2t \mid N$ (where $a \mid
b$ means $a$ divides $b$). Thanks to our divisibility hypotheses on $N$, it
follows that $u \mid 24$, and therefore $\phi$ must be a real character; in
other words, it takes on only the values $\pm 1$. Then the Eisenstein series
in $M_k(\Gamma_0(N))$ all have the form
\[
E^{\phi}_{t} = \frac{\delta(\phi)}{2} L(1-k, \phi) + \sum_{\substack{n \ge 1,\\ t \mid n}} \phi(n/t) \sigma_{k-1}(n/t) q^{n},
\]
where $\sigma_{\ell}(m) = \sum_{d\mid m} d^{\ell}$, $L(s,\phi)$ is the
$L$-function of $\phi$, and
\[
\delta(\phi) = \begin{cases} 1 & \text{ if $\phi$ is the trivial character of conductor $1$, and} \\
                             0 & \text{ otherwise.}
\end{cases}
\]
See, for example, Theorem~4.5.2 in \cite{DS2005}.

Since the Eisenstein series span the Eisenstein part of $M_k(\Gamma_0(N))$,
there exist constants $y^{\phi}_{t}$ such that
\begin{align*}
g_e  & = \sum_{t, \phi} y^{\phi}_{t} E^{\phi}_{t}\\
       & = e_0 + \sum_{t,\phi} \sum_{\substack{n \ge 1,\\ t \mid n}} y^{\phi}_{t} \phi(n/t) \sigma_{k-1}(n/t) \\
       & = e_0 + \sum_{n=1}^{\infty} \sum_{\substack{t \mid N,\\ t \mid n}} \left( \sum_{\phi}  y^{\phi}_t \phi(n/t)\right) \sigma_{k-1}(n/t).
\end{align*}
It is straightforward to check that whenever $t \mid n$,
\[
\frac{\sigma_{k-1}(n)}{\sigma_{k-1}(t)} \leq \sigma_{k-1}(n/t) \leq \frac{\sigma_{k-1}(n)}{t^{k-1}}.
\]
This implies that if we set
\begin{flalign*}
&& r_g(t,n) & = \begin{cases}
\frac{1}{t^{k-1}} & \text{ if $\sum_{\phi}  y^{\phi}_{t} \phi(n/t) < 0$, and}\\
 \frac{1}{\sigma_{k-1}(t)} & \text{ if $\sum_{\phi}  y^{\phi}_{t} \phi(n/t) \geq 0$,}
\end{cases} &&\\
&& r_g(n) & = \sum_{\substack{t \mid N,\\ t \mid n}} \left( \sum_{\phi}  y^{\phi}_{t} \phi(n/t) \right) r_g(t,n), &&\\
&\text{and}\hidewidth\\
&& r_g &= \min_{n \geq 1} r_g(n) = \min_{1 \leq n \leq N} r_g(n), &&
\end{flalign*}
then
\[
e_n \geq  \sigma_{k-1}(n) r_g \geq n^{k-1}r_g.
\]
This completes step~(2), provided that $r_g$ is positive. If it is not
positive, then our test will be inconclusive, since we are unable to certify
that even the Eisenstein part is nonnegative.

Combining the results of the previous two steps, we find that
\[
a_n \geq n^{k-1} r_g - n^{k/2} C_g.
\]
Since $k > 2$, this inequality provides an easily computed bound $Q =
\lfloor(C_g/r_g)^{2/(k-2)}\rfloor$ such that $a_n > 0$ for all $n > Q$.
Because of the large gap between $n^{k-1}$ and $n^{k/2}$, the bound $Q$ is
typically relatively small. Finally, to certify that the coefficients of $g$
are all nonnegative, we explicitly compute the coefficients $a_n$ for $n \leq
Q$.

This method will not always work, without more careful estimates.  For
example, it fails if $a_n$ is not eventually positive.  That can occur in
practice: in the example from Section~\ref{algorithm} with $d=16$, $N=4$, and
$T=2$, the optimal modular form is
\[
g = 1 + 4320q^2 + 61440q^3 + 522720q^4 + 2211840q^5 + 8960640q^6 + \dotsb,
\]
which has eventually positive coefficients, but
\[
\widetilde{g} =  16 + 69120q^4 + 983040q^6 + 8363520q^8 + 35389440q^{10} + \dotsb,
\]
which does not.  Thus, proving that $\widetilde{g}$ has nonnegative
coefficients requires a little more care.  However, we have not observed this
phenomenon for the best choices of $T$ in any of the cases we have examined.
If it were to occur, it could be handled by distinguishing between the values
of $r_g(n)$ for different residue classes of $n$ modulo $N$, and showing that
the cuspidal contribution vanishes whenever $r_g(n)=0$.

\section{Numerical results} \label{results}

Table~\ref{table:results} shows our numerical results. We used the SageMath
computer algebra system \cite{sage} for our calculations, with one exception:
we used Magma \cite{BCP1997} to compute bases for modular forms and the
action of the Atkin-Lehner involution.  This combination works conveniently,
because SageMath has an interface for calling Magma code.

To produce rigorous results, we used exact rational arithmetic, and we proved
nonnegativity of coefficients using the techniques of
Section~\ref{positivityCheck}. For calculations with forms of level $24$, we
directly solved the linear program over $\Q$; for level $96$, we instead used
floating point arithmetic to obtain an approximate solution, which we then
used to obtain a rational solution and prove its correctness and optimality.
All the numbers in the table are rounded correctly: lower bounds are rounded
down, and upper bounds are rounded up.
The data underlying the new bounds in Table~\ref{table:results} can be downloaded
from \cite{CT}. In the notation of Section~\ref{algorithm}, this data set contains the
$q$-expansion coefficients for $g(z) = \sum_n a_n q^n$
and $\widetilde{g}(z) = \sum_n b_n q^n$ with $0 \le n < 500$, which is enough information
to determine these modular forms uniquely.

\begin{table}
\caption{Center density bounds in dimensions $8$ through $32$.
The upper bound is the linear programming bound,
computed using the best auxiliary function currently known \cite{ACHLT}, while the dual
bound is based on the given values of $N$ and $T$, and the record packing is the densest
packing currently known \cite{CS1999}.  In dimensions $12$ and $16$,
we include both $N=96$ and $N=24$ for comparison.}
\label{table:results}
\begin{tabular}{llllrr}
\toprule
Dimension & \ Record packing & \ \ Dual bound & \ \ Upper bound & $N$ & $T$\\
\midrule
$8 $      & $\phantom{0}0.0625$           &  $\phantom{00}0.0625$                  & $\phantom{00}0.0625$               & $1$  & $1$\\
$12$      & $\phantom{0}0.037037$          & $\phantom{00}0.062446$          & $\phantom{00}0.062742$              & $96$ & $9$\\
     &           & $\phantom{00}0.059781$          &               & $24$ & $4$\\
$16$      & $\phantom{0}0.0625$           & $\phantom{00}0.106284$          & $\phantom{00}0.107059$              & $96$ & $20$\\
      &            & $\phantom{00}0.103948$          &              & $24$ & $6$\\
$20$      & $\phantom{0}0.131537$          & $\phantom{00}0.260996$          & $\phantom{00}0.276169$              & $24$ & $9$\\
$24$      & $\phantom{0}1$              &  $\phantom{00}1$                  & $\phantom{00}1$                  & $1$ & $2$ \\
$28$      & $\phantom{0}1$              & $\phantom{00}4.591741$          & $\phantom{00}4.828588$              & $24$ & $9$\\
$32$      & $\phantom{0}2.565784$           & $\phantom{0}28.086665$         & $\phantom{0}29.942182$             & $24$ & $12$\\
\bottomrule
\end{tabular}
\end{table}

\section{Open problems} \label{sec:open}

Our new lower bounds in Table~\ref{table:results} come fairly close to the
known upper bounds, but they do not agree to many decimal places.  We believe
that the upper bounds agree with the true linear programming bound, aside
from rounding the last decimal place up, while the lower bounds could be
further improved.  One difficulty in doing so is that modular forms are
inherently quantized: in the summation formula
\[
\sum_{n=0}^{\infty} a_{n} f(\sqrt{n}) =
\left(\frac{2}{\sqrt{N}}\right)^{d/2} \sum_{n=0}^{\infty} b_{n}
\widehat{f}\mathopen{}\left(\frac{2\sqrt{n}}{\sqrt{N}}\right)\mathclose{},
\]
there is no possibility to perturb the radii $\sqrt{n}$ or $2\sqrt{n/N}$
slightly, and so one must do the best one can using only radii of these
forms.  In particular, closely matching the upper bound may require $N$ to be
very large, perhaps on the order of $10^{10}$ if we wish to match ten digits,
and dealing with such large $N$ is not practical.  Any feasible method that
could close the gap between the primal and dual bounds to within a factor of
$1+10^{-10}$ would be a significant advance, and modular forms might not be
the right tool for this purpose.  For comparison, \cite{TS2006} and
\cite{SST2008} obtain dual linear programming bounds in high dimensions using
an entirely different approach.

Another topic we leave open is computations in dimensions that are not
divisible by $4$.  We see no theoretical obstacle to such an extension: one
must simply use modular forms of odd weight (for dimensions divisible by $2$
but not $4$) or half-integral weight (for odd dimensions), and replace
$\Gamma_0(N)$ with $\Gamma_1(N)$ so that such forms exist.  However, we have
not implemented these computations.  We have also not explored the
uncertainty principle introduced in \cite{BCK2010} and further studied in
\cite{CG2019}, for which one could again prove dual bounds using modular
forms.

One intriguing possibility that may be nearly within reach is proving that
there exists a dimension in which the linear programming bound is not sharp.
All dimensions except $1$, $2$, $8$, and $24$ seem to have this property, but
so far no proof is known. Three dimensions would be a natural target, because
we know the optimal packing density, and thus it would suffice to prove any
dual bound greater than this density.  In higher dimensions, it would require
an improvement on the linear programming bound. The only such bound currently
known is Theorem~1.4 from de Laat, Oliveira, and Vallentin's paper
\cite{LOV2014}, which is a refinement of the linear programming bound that
seems to give a small numerical improvement in dimensions $3$, $4$, $5$, $6$,
$7$, and $9$ (see Table~1 in \cite{LOV2014}) and presumably higher dimensions
as well, aside from $24$. Any dual bound greater than this improved upper
bound would suffice to show that the linear programming bound is not sharp.
Conversely, it would be interesting to prove dual bounds for the theorem of
de Laat, Oliveira, and Vallentin itself.

\end{document}